\documentclass[reqno]{amsart}
\usepackage{amssymb,amsmath,amsthm,mathrsfs, array}
\usepackage{amsfonts,latexsym,amsxtra,amscd}
\usepackage{pgfplots, tikz}
\usepackage{hyperref}
\usepackage{graphicx, stmaryrd}
\usepackage[centering]{geometry}

\def\bc{\begin{center}}
	\def\ec{\end{center}}
\def\be{\begin{equation}}
	\def\ee{\end{equation}}

\def\N{\mathbb N}

\def\R{\mathbb R}

\newcommand\hdim{\dim_{\mathrm H}}

\newtheorem{thm}{Theorem}[section]
\newtheorem{prop}[thm]{Proposition}
\newtheorem{cor}[thm]{Corollary}
\newtheorem{lem}[thm]{Lemma}

\newtheorem{rem}[thm]{Remark}

\numberwithin{equation}{section}

\begin{document}
	
	\title[ Irrationality Exponents and Partial Quotient Growth in Continued Fractions]
	{Irrationality Exponents and Partial Quotient Growth in Continued Fractions}
	

	\author[W. Cheng]{Wanjin Cheng}
	\address[Wanjin Cheng]{School of Mathematics, South China University of Technology, Guangzhou, 510640, China}
	\email{chengwj0227@163.com}
	\author[J. Feng]{Jing Feng$^{\ast}$}
	\address[Jing Feng]{School of Mathematics and Systems Science, \& Center for Mathematical Sciences, Wuhan University of Science and Technology, Wuhan, 430081, P. R. China}
	\email{jing.feng@wust.edu.cn}
	\thanks{$^{\ast}$Corresponding author.}
	\date{}
	\maketitle{}
	
	\begin{abstract}
		Let $[a_1(x),a_2(x),\ldots,a_n(x),\ldots]$ be the continued fraction expansion of irrational $x\in[0,1)$, and let $q_n(x)$ be the denominator of the $n$-th convergent. In this paper, we study how the growth rate of $a_{n+1}(x)$ on a prescribed logarithmic size interacts with its upper growth rate relative to $q_n(x)$. For $\psi:\mathbb N\to\mathbb{R}_{\ge0}$ satisfying $\psi(n)\to \infty$ and $\alpha, \beta\in [0, \infty]$,  define the joint level set
		\[F_{\alpha,\beta}:=\Big\{x\in [0,1)\colon \liminf_{n\to\infty}\frac{\log (a_{n+1}(x))}{\psi(n)}=\alpha,\  \limsup_{n\to\infty}\frac{\log (a_{n+1}(x))}{\log q_n(x)}=\beta\Big\}.
		\]
		We determine the Hausdorff dimension of $F_{\alpha,\beta}$ for all values of $ \alpha$ and $\beta$. Our results is related to several earlier results on the metric theory of continued fractions, including those of Bugeaud [Math. Ann. {327} (2003)] on irrationality exponents, Wang--Wu [Adv. Math. {218} (2008)] on the growth of partial quotients, and Song--Tan--Zhang [Nonlinearity {37} (2024)] on the joint distribution of convergence and irrationality exponents.
	\end{abstract}

	\bigskip
	\section{Introduction}
	Metric Diophantine approximation studies the size of sets of
	real numbers that admit rational approximations of a prescribed quality. As a direct consequence of Dirichlet's theorem, for every irrational number $x\in\R$,
	\begin{equation}\label{dirichlet cor}
		\Bigl|x-\frac pq\Bigr|<\frac1{q^2}
	\end{equation}
	has infinitely many solutions $(p,q)\in\mathbb Z\times\mathbb N$. A fundamental metric counterpart is  Khintchine theorem \cite{Khinchin24}: if $\psi: \mathbb{N}\to\R_{\ge0}$ is an  increasing function with $\psi\to\infty$, the set of $x\in\R$ for which
	\[\Bigl|x-\frac pq\Bigr|<\frac {1}{q\psi(q)}\] holds infinitely often has either zero or full Lebesgue
	measure according as $\sum_{q\geq1} \frac{1}{ \psi(q)}$ converges or diverges. For any $v\ge2$, Jarn\'ik \cite{Jarnik29} and Besicovitch \cite{Besicovitch} independently proved that the Hausdorff dimenison of the set
	\[
	W_v
	:=\biggl\{x\in[0,1):
	\Bigl|x-\frac pq\Bigr|<\frac{1}{q^{v}}~
	\text{ for infinitely many }(p,q)\in\mathbb Z\times\mathbb N
	\biggr\},
	\]
	is 	$\frac2v$.
	This leads to the study of the irrationality exponent
	\begin{equation}\label{irr exponent}
		v(x):=\sup\left\{v>0:
		x\in W_v
		\right\}.
	\end{equation}
	By \eqref{dirichlet cor}, $v(x)\geq2$ for every irrational $x$, while $v(x)=2$ for Lebesgue
	almost every $x\in [0,1)$.

	Continued fractions provide an important tool in metric Diophantine approximation. Denote by
	\[
	x=[a_1(x),a_2(x),\ldots]
	\]
	the continued fraction expansion of an irrational
	$x\in[0,1)$, and let $p_n(x)/q_n(x)$ be its $n$-th convergent. The standard estimates (\cite{Khi64})
	\[
	\frac{1}{(a_{n+1}(x)+2)q^2_n(x)}
	<\left|x-\frac{p_n(x)}{q_n(x)}\right|
	<\frac{1}{a_{n+1}(x)q^2_n(x)}
	\]
	show that the irrationality exponent defined in \eqref{irr exponent} can be represented by
	\begin{equation}\label{irr-limsup}
		v(x)
		=2+\limsup_{n\to\infty}
		\frac{\log a_{n+1}(x)}{\log q_n(x)}.
	\end{equation}
	Good \cite{Good41} showed that for every \(\beta>0\),
	\begin{equation}\label{JARNIK-1}
		{J}_\beta:=\Big\{x\in [0,1):\  a_{n+1}(x)\geq q^\beta_n(x) \ \text{for infinitely many }n\in \mathbb{N}\Big\},
	\end{equation}
	has Hausdorff dimension $\frac{2}{\beta+2}$. By a result of Bugeaud \cite{Bugeaud}, the corresponding exact level set
	\begin{equation}\label{dim of Fbeta}
		F_{\beta}:=\biggl\{x\in[0,1):\limsup_{n\to\infty}
		\frac{\log a_{n+1}(x)}{\log q_n(x)}=\beta\biggr\}
	\end{equation}
	has the  Hausdorff dimension $\frac{2}{\beta+2}$. Another characteristic quantity associated with the partial quotients is their convergence exponent.  For an irrational number
	$x=[a_1(x),a_2(x),\ldots]\in[0,1)$, it is defined by
	\[
	\tau(x)
	:=
	\inf\biggl\{
	s\geq0:
	\sum_{n=1}^{\infty}a^{-s}_n(x)<\infty
	\biggr\}.
	\]
	The convergence exponent describes the distributional growth of the
	partial quotients. From a result of Khintchine \cite{Khinchin24}, $\tau(x)=\infty$ for Lebesgue almost every
	$x\in[0,1)$. Moreover, it was proved by Fang et al. \cite{Fang22} that, for every
	$0\leq\tau<\infty$,
	\[
	\dim_{\rm H}\left\{
	x\in(0,1):\tau(x)=\tau
	\right\}=\frac{1}{2}.
	\]
	Here and below $\dim_{\rm{H}}$ denotes the Hausdorff dimension. For $\tau\in[0,\infty]$ and $v\geq2$, Song, Tan and Zhang \cite{Song} subsequently investigated the joint
	level sets
	\[E_{\tau,v}:=\left\{
	x\in(0,1):
	\tau(x)=\tau,~ v(x)=v
	\right\}.\]
	Their result is stated as follows.
	\begin{thm}[{\cite[Theorem 1.1]{Song}}] \label{Song-thm}
		For every $\tau\in[0,\infty]$ and $v\geq2$, we have
		\[
		\dim_{\rm H}E_{\tau,v}
		=
		\begin{cases}
			\dfrac{2}{v},&\tau=\infty;\\[2mm]
			\dfrac{1}{v},&0\leq\tau<\infty.
		\end{cases}
		\]
	\end{thm}
	
	A second classical question concerns the growth of $a_n(x)$. The Borel--Bernstein theorem \cite{Bernstein,Borel} states that,  for any positive function $\psi$, the set of $x\in[0,1)$ for which
	\[
	a_n(x)\geq\psi(n)\quad\text{for infinitely many }n\in\N
	\]
	has full or zero Lebesgue measure according as
	$\sum_{n\geq1}\frac{1}{\psi(n)}$ diverges or converges. The Hausdorff dimension of related exceptional sets was studied by Good \cite{Good41}, {\L}uczak \cite{Luczak}, and Wang--Wu \cite{WW08}.
	Further results on the growth of partial quotients can be found in \cite{Hass,Liao16, WX,Xu}.
	As an immediate consequence of the Borel--Bernstein theorem,
	\[
	\limsup_{n\to\infty}\frac{\log a_n(x)}{\log n}=1
	\]
	for Lebesgue almost every $x$. More recently, Fang, Ma and Song \cite{Fang} studied the exceptional sets obtained by
	comparing $\log a_n(x)$ with the function
	$\psi$ satisfying $\psi(n)\to\infty$. They analyzed the Hausdorff dimensions of the corresponding limit, lower--limit and upper--limit level sets, and showed that these three spectra can be genuinely different. Put
	\[B_\psi
	:=\limsup_{n\to\infty}\psi(n)^{1/n}
	=\exp\biggl(\limsup_{n\to\infty}\frac{\log\psi(n)}n\biggr)
	\in[1,\infty].
	\]
	They \cite[Theorem 1.3]{Fang} proved that the Hausdorff dimension of the set
	\[
	G_{1}:=\biggl\{x\in[0,1):
	\liminf_{n\to\infty}\frac{\log a_{n+1}(x)}{\psi(n)}=1
	\biggr\}
	.	\]
	is equal to $\frac1{1+B_\psi}$.
	
	We study the partial quotient \(a_{n+1}(x)\) on two different scales: the prescribed logarithmic scale $\psi(n)$ and the intrinsic scale $\log q_n(x)$, with the latter determining the irrationality exponent.
	We ask when prescribed levels on these two scales can occur simultaneously and how imposing both conditions affects the Hausdorff dimension. For $0\leq\alpha\leq\infty$ and
	$0\leq\beta\leq\infty$, we consider the joint level set
	\begin{equation*}
		F_{\alpha,\beta}
		:=\left\{x\in(0,1):
		\liminf_{n\to\infty}
		\frac{\log a_{n+1}(x)}{\psi(n)}=\alpha,
		~
		\limsup_{n\to\infty}
		\frac{\log a_{n+1}(x)}{\log q_n(x)}=\beta
		\right\}.
	\end{equation*}
	
	Our main result gives a complete answer.
	\begin{thm}\label{thm}
		Let $\psi:\mathbb N\to(0,\infty)$ satisfy $\psi(n)\to\infty$ as $n\to\infty$, and define
		\[
		B_{\psi}
		:=
		\limsup_{n\to\infty}\psi(n)^{1/n}
		=
		\exp\Bigl(
		\limsup_{n\to\infty}
		\frac{\log\psi(n)}{n}
		\Bigr)
		\in[1,\infty].
		\]
		Then
		\[
		\dim_{\rm H}F_{\alpha,\beta}
		=
		\begin{cases}
			\displaystyle \frac{2}{2+\beta},
			&
			\alpha=0,\quad 0\leq\beta<\infty;
			\\[8pt]
			\displaystyle \frac{1}{2+\beta},
			&
			0<\alpha\leq \infty, \quad0\leq\beta<\infty
			\quad\text{and}\quad
			B_\psi\leq1+\beta;
			\\[8pt]
			0,
			&
			0<\alpha\leq \infty,\quad 0\leq\beta<\infty
			\quad\text{and}\quad
			B_\psi>1+\beta;
			\\[6pt]
			0,
			&
			\beta=\infty.
		\end{cases}
		\]
	\end{thm}
	The zero--dimensional case in the third line has a stronger conclusion.
	\begin{prop}\label{empty}
		Assume  $0< \alpha\leq \infty$ and $0\leq \beta <\infty$. If $B_{\psi}> 1+\beta$, then
		\[F_{\alpha, \beta}=\emptyset.\]
	\end{prop}
	
	Theorem~\ref{thm} reveals two features of the joint spectrum. For finite $\beta$, if $\alpha=0$, the lower--limit condition does not reduce the dimension.
	If $\alpha>0$, the set $F_{\alpha, \beta}$ has positive dimension when $B_\psi\leq 1+\beta$. If $1< B_\psi<\infty$, the critical value
	\[
	\beta_c:=B_\psi-1
	\]
	separates an empty set ($\beta<\beta_c$) from a positive--dimensional set ($\beta_c\leq \beta<\infty$).
	\begin{prop}
		Suppose $0<\alpha<\infty$ and $1\leq B_\psi<\infty$.
		Then for every $\beta_c\leq\beta<\infty$,
		\[
		\dim_{\rm H}F_{\alpha,\beta}
		=
		\frac{1}{2}\dim_{\rm H}F_\beta.
		\]
	\end{prop}
	\begin{rem}
		Suppose $0<\alpha<\infty$ and $1\leq B_\psi<\infty$. The joint dimension equals the smaller of these two dimensions at $\beta=\beta_c$, but is strictly smaller than $G_{\alpha}$ and $F_\beta$ when $\beta_c< \beta<\infty$.
	\end{rem}
	
	As a consequence of Theorem \ref{thm}, together with Theorem \ref{Song-thm},  we obtain the following sub--level set analogue: replacing the exact condition \(\tau(x)=\tau\) by \(\tau(x)\leq\tau\) does not change the Hausdorff dimension. Define
	\[
	E_{\le\tau,v}:=\left\{
	x\in(0,1):
	\tau(x)\leq\tau,\ v(x)=v
	\right\}.
	\]
	Then, we have
	\begin{cor}\label{cor-Song}
		For every $\tau\in[0,\infty]$ and $v\geq2$, we have
		\[
		\dim_{\rm H}E_{\le\tau,v}
		=\dim_{\rm H}E_{\tau,v}=
		\begin{cases}
			\dfrac{2}{v},&\tau=\infty;\\[2mm]
			\dfrac{1}{v},&0\leq\tau<\infty.
		\end{cases}
		\]
	\end{cor}
	We also state a sparse--insertion principle for
	continued fractions in this section. It will be proved and applied in the treatment
	of the  case \(\alpha=0\) in
	Theorem \ref{thm}.
	\begin{prop}
		\label{lem-insertion}
		Let $\mathcal N=\{N_1<N_2<\cdots\}\subset\mathbb N$ satisfy
		\[
		r(n):=\#\bigl(\mathcal N\cap[1,n]\bigr)=o(n).
		\]
		Let $\{c_k\}_{k\geq1}$ be a bounded sequence of positive
		integers. For
		$x=[a_1,a_2,\ldots]\in[0,1)\setminus\mathbb Q$, define
		$\iota(x)=[b_1,b_2,\ldots]$
		by inserting $c_k$ at position $N_k$ and keeping all the
		original digits in their original order. Then, for every
		$E\subset[0,1)\setminus\mathbb Q$,
		\[
		\dim_{\mathrm H}\iota(E)
		=
		\dim_{\mathrm H}E.
		\]
	\end{prop}
	This paper is organized as follows. In Section \ref{section2}, we introduce the notation and preliminaries used throughout the paper. In Section \ref{section3},  for the case $0<\alpha\leq\infty
	~\text{and}~
	0\leq\beta<\infty $, we establish the upper bound in  Theorem \ref{thm} for the case $0<\alpha\leq \infty$, $0\leq\beta<\infty$, and construct Cantor subsets of \(F_{\alpha,\beta}\) to obtain the  lower bound.
	In Section \ref{The remaining cases}, we establish the  remaining cases of  Theorem \ref{thm} and we prove that
	inserting uniformly bounded prescribed partial quotients at positions
	of zero asymptotic density preserves the Hausdorff dimension. In the last Section, we establish connections between our results and the convergence and irrationality exponents.

	\section{Preliminaries}\label{section2}
	The  continued fraction expansion is generated by the Gauss
	transformation
	\[
	T:[0,1)\longrightarrow[0,1),
	\qquad
	T(0):=0,
	\qquad
	T(x):=\frac1x-\left\lfloor\frac1x\right\rfloor
	\quad\text{for }x\in(0,1).
	\]
	Every irrational number $x\in(0,1)$ admits a unique infinite regular
	continued fraction expansion
	\begin{align}\label{CFE}
		x
		&=
		\cfrac{1}{
			a_1(x)+
			\cfrac{1}{
				a_2(x)+
				\cfrac{1}{
					\ddots+
					\cfrac{1}{
						a_n(x)+\ddots
					}
				}
			}
		}.
	\end{align}
	Here
	\(
	a_n(x)
	=
	\left\lfloor
	\frac{1}{T^{n-1}(x)}
	\right\rfloor,
	~ n\geq1,
	\)
	where $\lfloor z\rfloor$ denotes the greatest integer not exceeding
	$z$. The integer $a_n(x)$ is called the \emph{$n$-th partial quotient}
	of $x$. For brevity, we write \eqref{CFE} as
	\(
	x=[a_1(x),a_2(x),\ldots].
	\)
	For each $n\geq1$, the finite truncation
	\[
	\frac{p_n(x)}{q_n(x)}
	:=
	[a_1(x),a_2(x),\ldots,a_n(x)]
	\]
	is called the \emph{$n$-th convergent} of $x$. As usual, we adopt the
	initial conditions
	\[
	p_{-1}(x)=1,\quad p_0(x)=0,
	\quad
	q_{-1}(x)=0,\quad q_0(x)=1.
	\]
	Then the numerators $\{p_n(x)\}_{n\geq1}$ and denominators
	$\{q_n(x)\}_{n\geq1}$ satisfy the recurrence relations
	\cite{Khi64}
	\begin{equation}\label{sequence-qn}
		\begin{cases}
			p_n(x)=a_n(x)p_{n-1}(x)+p_{n-2}(x),\\[2mm]
			q_n(x)=a_n(x)q_{n-1}(x)+q_{n-2}(x),
		\end{cases}
		\qquad n\geq1.
	\end{equation}
	
	For any $n\geq1$ and any
	$(a_1,\ldots,a_n)\in\mathbb N^n$, the \emph{cylinder of order $n$}
	is defined by
	\[
	I_n(a_1,\ldots,a_n)
	:=
	\left\{
	x\in[0,1):
	a_1(x)=a_1,\ldots,a_n(x)=a_n
	\right\}.
	\]
	Since $p_n(x)$ and $q_n(x)$ depend only on the first $n$ partial
	quotients of $x$, they are constant on each cylinder
	$I_n(a_1,\ldots,a_n)$. Thus, for simplicity, we
	write $a_n$, $p_n$, and $q_n$ in place of
	$a_n(x)$, $p_n(x)$, and $q_n(x)$, respectively.
	
	We shall repeatedly use the following  properties of
	continued fraction cylinders.
	
	\begin{prop}[\cite{Khi64}]\label{length of cylinder}
		Let $n\geq1$ and $(a_1,\ldots,a_n)\in\mathbb N^n$, and let
		$p_n/q_n=[a_1,\ldots,a_n]$, where $p_n$ and $q_n$ are defined by
		\eqref{sequence-qn}. Then the following statements hold.
		
		\begin{enumerate}
			\item
			The cylinder $I_n(a_1,\ldots,a_n)$  takes the form
			\[
			I_n(a_1,\ldots,a_n)
			=
			\begin{cases}
				\displaystyle
				\left[
				\frac{p_n}{q_n},
				\frac{p_n+p_{n-1}}{q_n+q_{n-1}}
				\right),
				& n \text{ is even};\\[4mm]
				\displaystyle
				\left(
				\frac{p_n+p_{n-1}}{q_n+q_{n-1}},
				\frac{p_n}{q_n}
				\right],
				& n \text{ is odd}.
			\end{cases}
			\]
			Its length satisfies
			\[
			\frac{1}{2q_n^2}
			\leq|I_n(a_1,\ldots,a_n)|
			=
			\frac{1}{q_n(q_n+q_{n-1})}\leq
			\frac{1}{q_n^2}.
			\]
			
			\item
			The denominator $q_n$ satisfies $q_n\geq 2^{(n-1)/2}$ and
			$\prod_{k=1}^n a_k \leq q_n \leq 2^n\prod_{k=1}^n a_k$.
			
			\item
			It holds that
			\[\frac{a_k+1}{2}\leq\frac{q_n(a_1,\dots,a_n)}{q_{n-1}(a_1,\dots,a_{k-1},a_{k+1},\dots,a_n)}\leq a_{k}+1.
			\]
		\end{enumerate}
	\end{prop}
	For further basic properties of continued fractions, we refer
	the reader to \cite{A14,Khi64,RS92}. The following lemma is useful for
	establishing lower bounds for the Hausdorff dimensions of certain sets
	defined in terms of continued fraction expansions.
	\begin{lem}[{\cite[Lemma 3.2]{ALBW13}}]\label{3}
		Let $\{s_n\}_{n\geq1}$ be a sequence of positive integers tending to infinity. Then, for every $N\geq2$,
		\begin{equation*}
			\begin{aligned}
				&\hdim \big\{x\in[0,1): s_n\leq a_n(x)< Ns_n,\  \text{for all}\ n \geq 1\big\}\\
				&=\liminf_{n\to\infty}\frac{\log (s_1s_2\cdots s_n)}{2\log (s_1s_2\cdots s_n)+\log s_{n+1}}\\
				&=\frac{1}{2+\limsup\limits_{n\to\infty}\frac{\log s_{n+1}}{\log (s_1s_2\cdots s_n)}}.
			\end{aligned}
		\end{equation*}
	\end{lem}
	
	\begin{lem}[{\cite[Remark 2]{ALBW13}}]\label{lem-4}
		For any $\beta>0$, the Hausdorff dimension of the set
		\[J_\beta^*:=\bigg\{x\in J_\beta:\ \lim_{n\to\infty}\frac{\log q_n(x)}{n}=\infty\bigg\},\]
		is $\frac{1}{\beta+2}$, where $J_\beta$ is defined in \eqref{JARNIK-1}.
	\end{lem}
	
	\section{Proof of  Theorem \ref{thm}: the case $0<\alpha\leq\infty
		~\text{and}~
		0\leq\beta<\infty$}\label{section3}
	We first determine,  for $\alpha>0$, the range of $0\leq\beta<\infty$ for which $F_{\alpha,\beta}$ is nonempty. Throughout
	this section, we assume that
	\[
	0<\alpha\leq\infty
	\quad\text{and}\quad
	0\leq\beta<\infty.
	\]
	\subsection{A necessary condition}
	
	We begin by establishing a necessary condition for the non--emptiness of
	\(F_{\alpha,\beta}\).
	\begin{proof}[Proof of Proposition \ref{empty}]
		Suppose that $x\in F_{\alpha,\beta}$, we show that necessarily
		$B_\psi\leq1+\beta$.
		
		First, assume that $0<\alpha<\infty$. For every
		$\varepsilon>0$, there exists \(N_1:=N_1(x,\varepsilon)\) such that
		\begin{equation}\label{relationship of two}
			(\alpha-\varepsilon)\psi(n)
			\leq
			\log a_{n+1}(x)
			\leq
			(\beta+\varepsilon)\log q_n(x).
		\end{equation}
		for every \(n\geq N_1\).
		It follows from \eqref{sequence-qn} that, for all
		\(n\ge N_1\),
		\[
		\begin{split}
			\log q_{n+1}(x)
			&\leq
			\log (2a_{n+1}(x)q_n(x)) \leq
			(1+\beta+\varepsilon)\log q_n(x)+\log 2.
		\end{split}
		\]
		Since \(q_n(x)\ge 2^{\frac{n-1}{2}}\), the last term can be absorbed into
		\(\varepsilon\log q_n(x)\). Hence, we have
		\[
		\log q_{n+1}(x)
		\leq
		(1+\beta+2\varepsilon)\log q_n(x)
		\]
		for every \(n\geq N_1\). Iterating this inequality, we obtain
		\begin{equation}
			\label{qn-upper}
			\log q_n(x)
			\leq
			(1+\beta+2\varepsilon)^{n-N_1}\log q_{N_1}(x)
		\end{equation}
		for all  \(n\ge N_1\).
		Combining \eqref{sequence-qn},
		\eqref{relationship of two}, and
		\eqref{qn-upper}, we find that
		\[
		(\alpha-\varepsilon)\psi(n)
		\leq
		(\beta+\varepsilon)\log q_n(x)
		\leq
		(1+\beta+2\varepsilon)^{n-N_1}\log q_{N_1}(x)
		\]
		for all  $n\geq N_1$. Therefore,
		\[
		B_{\psi}
		=
		\limsup_{n\to\infty}\psi(n)^{1/n}
		\leq
		1+\beta+2\varepsilon.
		\]
		Letting $\varepsilon\to0$, we conclude that
		\[
		B_\psi\leq1+\beta.
		\]
		If $\alpha=\infty$, then for any fixed $M>0$ and $\varepsilon>0$, there exists \(N_2:=N_2(x, M, \varepsilon)\) such that
		\[M\psi(n)
		\leq
		\log a_{n+1}(x)
		\leq
		(\beta+\varepsilon)\log q_n(x)\]
		for all $n\geq N_2$. Repeating the argument above, with $M$ in place of
		$\alpha-\varepsilon$, again gives
		\[
		B_\psi\leq1+\beta.
		\]
		Therefore, if $B_\psi>1+\beta$, then
		$F_{\alpha,\beta}=\varnothing$.

	\end{proof}
	
	\subsection{The upper bound}
	The following proposition gives the upper bound required for Theorem~\ref{thm}.
	\begin{prop}
		For \(0<\alpha\leq\infty\) and \(0\leq\beta<\infty\),
		\[
		\dim_{\mathrm H}F_{\alpha,\beta}
		\leq
		\frac{1}{2+\beta}.
		\]
	\end{prop}
	
	\begin{proof}
		Fix \(c=\frac{\alpha}{2}\) if \(0<\alpha<\infty\), and set \(c=1\) if
		\(\alpha=\infty\). For every \(x\in F_{\alpha,\beta}\), there exists \(N:=N(x,c)\) such that
		\begin{equation}\label{an+1}
			\log a_{n+1}(x)\geq c\psi(n)
		\end{equation}
		for all  \(n\ge N\).  Set
		\(
		M_n:=\lfloor \exp\bigl(c\psi(n)\bigr)\rfloor.
		\)
		Since \(\psi(n)\to\infty\), we have \(M_n\to\infty\).

		Let $s>\frac{1}{2+\beta}$.
		Choose \(\varepsilon>0\) sufficiently small that
		\begin{equation}\label{def-r}
			r:=s(2+\beta-\varepsilon)>1.
		\end{equation}
		The second defining condition of $F_{\alpha,\beta}$ implies
		\begin{equation}
			\label{an>qn}
			a_{n+1}(x)\geq q^{\beta-\varepsilon}_n(x)
		\end{equation}
		for infinitely many \(n\in\N\).
		For \(N\in\mathbb N\) and
		\(\boldsymbol a=(a_1,\ldots,a_N)\in\mathbb N^N\), define
		\[
		F(N,\boldsymbol a)
		:=
		\left\{
		x\in F_\beta:
		(a_1(x),\ldots,a_N(x))=\boldsymbol a,\
		a_{n+1}(x)\geq M_n
		\text{ for all }n\geq N
		\right\}.
		\]
		By \eqref{an+1},
		\[
		F_{\alpha,\beta}
		\subset
		\bigcup_{N=1}^{\infty}
		\bigcup_{\boldsymbol a\in\mathbb N^N}
		F(N,\boldsymbol a).
		\]
		
		For \(n\geq N\), let
		\[
		\mathcal A_n(N,\boldsymbol a)
		:=
		\left\{
		(\boldsymbol a,a_{N+1},\ldots,a_n)\in\mathbb N^n:
		a_j\geq M_{j-1},\
		N+1\leq j\leq n
		\right\}.
		\]
		For \(\boldsymbol{b}\in\mathcal A_n(N,\boldsymbol{a})\),  define
		\(
		A_n(\boldsymbol{b})
		:=
		\max\left\{
		1,
		\left\lfloor  q^{\beta-\varepsilon}_n(\boldsymbol{b})\right \rfloor
		\right\},
		\)
		and
		\[
		J_n(\boldsymbol{b})
		:=
		\bigcup_{a\geq A_n(\boldsymbol{b})}
		I_{n+1}(a_1,\ldots,a_n,a).
		\]
		By \eqref{an>qn}, for every $m\geq N$,
		\[
		F(N,\boldsymbol{a})
		\subseteq
		\bigcup_{n=m}^{\infty}
		\ \bigcup_{\mathbf b\in\mathcal A_n(N,\boldsymbol{a})}
		J_n(\boldsymbol{b}).
		\]
		Using Proposition \ref{length of cylinder},
		we obtain
		\begin{equation}\label{cover}
			\begin{aligned}
				\bigl|J_n(\boldsymbol{b})\bigr|
				&\le
				\frac{1}{q^2_n(\boldsymbol{b})}
				\sum_{a\geq A_n(\boldsymbol{b})}\frac{1}{a^2}
				\notag\le
				q^{-(2+\beta-\varepsilon)}_n(\boldsymbol{b}),
		\end{aligned}\end{equation}
		where the last inequality also holds when \(\beta<\varepsilon\), since in this case
		\(A_n(\boldsymbol{b})=1\) and
		\(
		q_n^{-2}\leq q_n^{-(2+\beta-\varepsilon)}.
		\)
		
		It remains to estimate the \(s\)--dimensional cost of this cover.
		By \eqref{def-r} and (2) of Proposition \ref{length of cylinder}, we conclude that
		\[
		\begin{aligned}
			\sum_{\boldsymbol{b}\in\mathcal A_n(N,\mathbf a)}
			|J_n(\boldsymbol{b})|^s
			&\leq
			\sum_{\boldsymbol{b}\in\mathcal A_n(N,\mathbf a)}
			q^{-r}_n(\boldsymbol{b})\leq
			q^{-r}_N(\boldsymbol{a})
			\prod_{j=N+1}^{n}
			\Bigl(
			\sum_{a\geq M_{j-1}}a^{-r}
			\Bigr).
		\end{aligned}
		\]
		Since $r>1$ and $M_n\to\infty$, there exists $N_0$ such that
		\[
		\sum_{a\geq M_{j-1}}a^{-r}\leq\frac12
		\]
		for every $j\geq N_0+1$. We may restrict attention to
		\(N\geq N_0\). Hence, for $n\geq N\geq N_0$,
		\[
		\sum_{\boldsymbol{b}\in \mathcal A_n(N,\mathbf a)}
		|J_n(\boldsymbol{b})|^s
		\leq
		q^{-r}_N(\mathbf a)2^{-(n-N)}.
		\]
		Therefore, for every \(m\geq N\),
		\[
		\sum_{n=m}^{\infty}
		\sum_{\boldsymbol b\in\mathcal A_n(N,\boldsymbol a)}
		|J_n(\boldsymbol b)|^s
		\leq
		2^{s+1}
		q^{-r}_N(\boldsymbol a)
		2^{-(m-N)}.
		\]
		The right--hand side tends to \(0\) as \(m\to\infty\). Moreover,
		\(q_n\geq2^{(n-1)/2}\), so the diameters of the sets in the above
		covers tend to \(0\). It follows that
		\[
		\mathcal H^s\bigl(F(N,\boldsymbol a)\bigr)=0.
		\]
		Since \(F_{\alpha,\beta}\) is contained in a countable union of such
		sets,
		\[
		\mathcal H^s(F_{\alpha,\beta})=0.
		\]
		Since $s>\frac{1}{2+\beta}$ was arbitrary, the desired upper bound
		follows.
	\end{proof}
	
	\subsection{The lower bound}
	Assume $B_\psi\leq1+\beta$. To prove the lower bound, we construct a Cantor subset of $F_{\alpha,\beta}$ and estimate its Hausdorff dimension.
	\subsubsection{ Construction of a Cantor set.}\label{cantor set}
	We construct a sequence that will prescribe the logarithms of the partial
	quotients in the Cantor set used below.
	
	We first choose a sequence $\{\lambda_n\}_{n\geq1}$
	such that
	\begin{equation}\label{lambda-prop}
		\lambda_n>1+\beta,\quad
		\lim\limits_{n\to\infty}\lambda_n=1+\beta,
		\quad\text{and}\quad
		\lim\limits_{n\to\infty}\frac{\psi(n)}{\lambda_1\cdots\lambda_n}= 0.
	\end{equation}
	For completeness, such a sequence can be constructed as follows.  Let
	\[
	c_n:=\max\left\{\log\psi(n)-n\log (1+\beta),0\right\},
	\qquad
	\eta_n:=\sup_{m\geq n}\frac{c_m}{m}.
	\]
	The definition of $B_\psi$ gives $\eta_n\ge \eta_{n+1}$ for any $n\ge1$ and $\lim\limits_{n\to\infty}\eta_n=0$.  Set
	\[
	\delta_n:=\sqrt{\eta_n}+\frac{1}{\sqrt n},
	\quad
	\lambda_n:=(1+\beta)\mathrm{e}^{\delta_n}.
	\]
	Then $\lambda_n>1+\beta$ and $\lim\limits_{n\to\infty}\lambda_n= 1+\beta$.  Moreover, by the definition of $\eta_n$, we have
	$\eta_j\geq\eta_n\geq \frac{c_n}{n}$ whenever $j\leq n$. Hence, we have, for all
	sufficiently large $n$,
	\[
	\sum_{j=1}^n\delta_j
	\geq n\sqrt{\eta_n}+\sum_{j=1}^n\frac1{\sqrt j}
	\geq \sqrt{nc_n}+\sum_{j=1}^n\frac1{\sqrt j}
	\geq c_n+\sum_{j=1}^n\frac1{\sqrt j}.
	\]
	Consequently,
	\[
	\log\frac{\psi(n)}{\lambda_1\cdots\lambda_n}
	\leq c_n-\sum_{j=1}^n\delta_j\le -\sum_{j=1}^n\frac1{\sqrt j}
	\to-\infty,
	\]
	which proves \eqref{lambda-prop}.
	
	For $n\geq1$, define
	\begin{equation}\label{def-Ln}
		\ell_{n}:=\log L_{n}
		:=
		\alpha\sup_{m\geq n}
		\bigg\{
		\psi(m)\prod_{j=n}^{m-1}\lambda_j^{-1}
		\bigg\},
	\end{equation}
	where an empty product is defined as $1$. If $\alpha=\infty$, replacing \(\alpha\) with $n$ in the formula \eqref{def-Ln}, and it does not affect the subsequent operations. By
	\eqref{lambda-prop}, for every fixed $n$,
	\[
	\psi(m)\prod_{j=n}^{m-1}\lambda_j^{-1}
	=
	\frac{\lambda_m\psi(m)}{\lambda_1\cdots\lambda_m}
	\left(\lambda_1\cdots\lambda_{n-1}\right)
	\longrightarrow0.
	\]
	Thus, the supremum in \eqref{def-Ln} is finite and is attained.  Taking
	$m=n$ gives
	\begin{equation}\label{ell-properties1}
		\ell_n\geq\alpha\psi(n)(\alpha<\infty)\quad \text{and}\quad \ell_n\geq n\psi(n)(\alpha=\infty).
	\end{equation}
	Furthermore, we have
	\begin{equation}\label{ell-properties2}
		\ell_{n+1} \leq \lambda_n\ell_n(\alpha<\infty) \quad \text{and}\quad \ell_{n+1} \leq \frac{n+1}{n}\lambda_n\ell_n(\alpha=\infty).
	\end{equation}
	
	Write
	\[
	A_n:=\sum_{j=1}^{n-1}\ell_j,\quad n\geq2.
	\]
	Since $\psi(n)\to\infty$, it follows from
	\eqref{ell-properties1} that $A_n\to\infty$. We claim that
	\begin{equation}\label{ell-beta}
		\limsup_{n\to\infty}\frac{\ell_n}{A_n}\leq\beta.
	\end{equation}
	Indeed, by \eqref{ell-properties2}
	\[
	\ell_n-\ell_1
	=
	\sum_{j=1}^{n-1}(\ell_{j+1}-\ell_j)
	\leq
	\sum_{j=1}^{n-1}\Bigl(\frac{j+1}{j}\lambda_j-1\Bigr)\ell_j.
	\]
	For every $\varepsilon>0$, we have
	$\frac{j+1}{j}\lambda_j-1\leq\beta+\varepsilon$ for all
	$j\ge N$. Therefore,
	\[
	\ell_n
	\leq
	\sum_{j=1}^{N}\Bigl(\frac{j+1}{j}\lambda_j-1-\beta-\epsilon\Bigr)\ell_j
	+
	(\beta+\varepsilon)A_n.
	\]
	Dividing by $A_n$ and then
	letting $\varepsilon\to0$ proves \eqref{ell-beta}.
	Hence we  choose a positive sequence $\lim\limits_{n\to\infty}\theta_n=0$ such that
	\begin{equation}\label{ell-theta}
		\ell_n\leq(\beta+\theta_n)A_n,\quad n\geq2.
	\end{equation}
	
	We next locate infinitely many indices at which the envelope attains the prescribed scale. For every \(n\geq1\), let \(t_n\geq n\) be the smallest index at which the supremum in \eqref{def-Ln} is attained. If \(t_n>n\), then when passing from \(n\) to \(n+1\), every remaining candidate is multiplied by the same factor  \(\lambda_n\) if \(\alpha<\infty\), and \(\frac{n+1}{n}\lambda_n\) if \(\alpha=\infty\). Hence \(t_{n+1}=t_n\). Repeating this argument gives
	\[
	t_n=t_{n+1}=\dots=t_{t_n},
	\]
	the identity is immediate when \(t_n=n\).
	
	Let \(r_1<r_2<\dots\) be the distinct values of $\{t_n\}_{n\geq1}$, arranged in increasing order. This sequence is infinite because \(t_n\geq n\). Moreover, \(t_{r_k}=r_k\), so the term indexed by \(m=r_k\) attains the supremum in \eqref{def-Ln} at \(n=r_k\). Consequently,
	\begin{equation}\label{ell-r_k}
		\ell_{r_k}=
		\begin{cases}
			\alpha\psi(r_k),&\alpha<\infty,\\
			r_k\psi(r_k),&\alpha=\infty,
		\end{cases}
		\qquad k\geq1.
	\end{equation}
	Define the sparse
	set of exceptional positions
	\begin{equation*}
		\mathcal N:=\{N_k:k\geq1\},
		\qquad N_k:=r_{2^k}.
	\end{equation*}
	
	Define the sequence $\{u_n\}_{n\geq1}$ by
	\begin{equation}\label{define}
		u_n
		:=
		\begin{cases}
			L_n,
			& n\notin\mathcal N,\\[2mm]
			\max\{L_n,{(u_1\cdots u_{n-1})}^{\beta+\theta_n}\},
			& n\in\mathcal N.
		\end{cases}
	\end{equation}
	We define the Cantor subset as follows:
	\begin{equation*}
		A_{\alpha,\beta}:=\bigl\{x=[a_1,a_2,\ldots]: a_1=1,\
		\left \lceil u_n\right \rceil
		\leq a_{n+1}
		<2\left\lceil u_n\right\rceil
		\quad\text{for any }n\geq1\bigr\},
	\end{equation*}
	where \(\lceil x\rceil\) denotes the smallest integer greater than or equal to \(x\).	
	It remains to verify the subset relation.
	\begin{lem}\label{subset lem} Assume $B_{\psi}\leq 1+\beta$. For \(0<\alpha\leq\infty\) and \(0\leq\beta<\infty\), we have
		\[
		A_{\alpha,\beta}\subset F_{\alpha,\beta}.
		\]
	\end{lem}
	\begin{proof}
		Fix $x\in A_{\alpha,\beta}$.
		We first show that
		\begin{equation}\label{liminf condition}
			\liminf_{n\to\infty}\frac{\log (a_{n+1}(x))}{\psi(n)}=\alpha.
		\end{equation}
		Suppose first that $0<\alpha<\infty$. By
		\eqref{ell-properties1} and $u_n\geq L_n$,
		\[
		\log a_{n+1}(x)\geq\log u_n\ge\ell_n\geq\alpha\psi(n),
		\]
		and hence
		\begin{equation*}
			\liminf_{n\to\infty}
			\frac{\log a_{n+1}(x)}{\psi(n)}\geq\alpha.
		\end{equation*}
		There are infinitely many indices $j$ that are not powers of
		$2$. For such $j$, $r_j\notin\mathcal N$, and hence
		\[
		u_{r_j}=L_{r_j}.
		\]
		Since \(u_n\geq1\), the defining of
		\(A_{\alpha,\beta}\) give
		\[
		a_{n+1}(x)
		<
		2\lceil u_n\rceil
		\leq4u_n.
		\]
		Using \eqref{ell-r_k}, we obtain
		\[
		\frac{\log a_{r_j+1}(x)}{\psi(r_j)}
		\le
		\alpha+\frac{\log4}{\psi(r_j)}
		\leq\alpha.
		\]
		Since \(\psi(r_j)\to\infty\),
		\[
		\liminf_{n\to\infty}
		\frac{\log a_{n+1}(x)}{\psi(n)}
		\leq\alpha.
		\]
		Thus \eqref{liminf condition} holds when
		\(0<\alpha<\infty\).
		
		If \(\alpha=\infty\), then by \eqref{ell-properties1},
		\[
		\log a_{n+1}(x)
		\geq
		\log u_n
		\geq
		\ell_n
		\geq
		n\psi(n).
		\]
		Therefore,
		\[
		\liminf_{n\to\infty}
		\frac{\log a_{n+1}(x)}{\psi(n)}
		=
		\infty.
		\]
		This proves \eqref{liminf condition} in all cases.
		
		We next prove that
		\begin{equation}\label{limsup condition}\limsup_{n\to\infty}\frac{\log (a_{n+1}(x))}{\log q_n(x)}=\beta.
		\end{equation} For $n\geq2$,
		put
		\(
		S_n:=\sum_{j=1}^{n-1}\log u_j.\)
		Since $u_j\geq L_j$, we have $S_n\geq A_n$. If
		$n\notin\mathcal N$, then by \eqref{ell-theta}
		\begin{equation}\label{un-sn1}
			\frac{\log u_n}{S_n}
			=
			\frac{\ell_n}{S_n}
			\leq
			\frac{\ell_n}{A_n}
			\leq
			\beta+\theta_n.
		\end{equation}
		If $n\in\mathcal N$, then by \eqref{ell-theta}
		\[
		\ell_n
		\leq
		(\beta+\theta_n)A_n
		\leq
		(\beta+\theta_n)S_n.
		\]
		Hence the second term in the maximum
		\eqref{define} is dominant, and then
		\begin{equation}\label{un-sn2}
			\frac{\log u_n}{S_n}=\frac{(\beta+\theta_n)\log (u_1\cdots u_{n-1})}{S_n }
			=
			\beta+\theta_n.
		\end{equation}
		
		By part (2) of Proposition \ref{length of cylinder} and the bounds defining
		$A_{\alpha,\beta}$, we have
		\begin{equation}\label{qn-sn}
			S_n
			\leq
			\log q_n(x)
			\leq
			S_n+n\log4.
		\end{equation}
		Moreover, the defining bounds for $A_{\alpha,\beta}$ give
		\begin{equation}\label{an-un}
			0\leq\log a_{n+1}(x)-\log u_n<\log4.
		\end{equation}
		By \eqref{ell-properties1} and $\psi(n)\to\infty$, we have
		$\lim\limits_{n\to\infty}\frac{A_n}{n}=\infty$.
		Since $S_n\geq A_n$, it follows that
		$\lim\limits_{n\to\infty}\frac{S_n}{n}=\infty$.
		By \eqref{un-sn1}, \eqref{qn-sn} and \eqref{an-un}, we obtain
		\[
		\frac{\log a_{n+1}(x)}{\log q_n(x)}
		\leq
		\frac{\log u_n+\log4}{S_n}
		\leq
		\beta+\theta_n+\frac{\log4}{S_n}.
		\]
		Hence
		\[
		\limsup_{n\to\infty}
		\frac{\log a_{n+1}(x)}{\log q_n(x)}
		\leq\beta.
		\]
		
		On the other hand, along the subsequence
		\(n=r_{2^k}\in\mathcal N\), we have
		\[
		\log u_n=(\beta+\theta_n)S_n.
		\]
		Therefore, by \eqref{un-sn2}, \eqref{qn-sn} and \eqref{an-un},
		\[
		\frac{\log a_{n+1}(x)}{\log q_n(x)}
		\geq
		\frac{\log u_n}{S_n+3n\log2}
		=
		(\beta+\theta_n)
		\frac{S_n}{S_n+3n\log2}.
		\]
		Since \(\theta_n\to0\), we conclude that
		\[
		\limsup_{n\to\infty}
		\frac{\log a_{n+1}(x)}{\log q_n(x)}
		\geq\beta.
		\]
		This proves \eqref{limsup condition}. Therefore, Combining with \eqref{liminf condition} and \eqref{limsup condition}, we obtain $x\in F_{\alpha,\beta}$. The proof is complete.
	\end{proof}

	\subsubsection{The lower bound of $\hdim F_{\alpha,\beta}$}
	
	\begin{prop} For \(0<\alpha\leq\infty\) and \(0\leq\beta<\infty\), we have
		
		\[
		\dim_{\mathrm H}F_{\alpha,\beta}\geq\frac{1}{2+\beta}.
		\]
	\end{prop}
	
	\begin{proof}
		From Lemmas \ref{3} and \ref{subset lem}, we have
		\begin{equation*}
			\dim_{\mathrm H}F_{\alpha,\beta}\ge \dim_{\mathrm H}A_{\alpha,\beta}\ge \frac{1}{2+\limsup_{n\to\infty}\frac{\log u_{n+1}}{\log (u_1u_2\cdots u_n)}}.
		\end{equation*}
		By \eqref{un-sn1} and $\eqref{un-sn2}$, we conclude that
		\[
		\limsup_{n\to\infty}\frac{\log u_{n+1}}{\log (u_1u_2\cdots u_n)}=\beta.
		\]
		Then, we obtain
		\begin{equation*}
			\dim_{\mathrm H}F_{\alpha,\beta}\ge \frac{1}{2+\beta}.
		\end{equation*}
	\end{proof}
	\section{Proof of  Theorem \ref{thm}: the remaining cases}\label{The remaining cases}
	\subsection{The case $\alpha=0$}
	We first give a proof of Proposition \ref{lem-insertion}. This allows us to impose the required lower--limit condition without affecting the dimension.
	\begin{proof}[Proof of Proposition \ref{lem-insertion}]

		We shall use the faithful covering family introduced in
		\cite[Theorem~2.2]{Liu16}. For $n\geq1$, let $\mathcal A_n$
		consist of all fundamental intervals of order $n$ of the form
		\[
		J_n(a_1,\ldots,a_{n-1};s,t)
		:=
		\bigcup_{d=s}^{t}
		I_n(a_1,\ldots,a_{n-1},d),
		\qquad 1\leq s\leq t,
		\]
		and put
		\(
		\mathcal A:=\bigcup_{n\geq1}\mathcal A_n.
		\)
		By \cite[Theorem~2.2]{Liu16}, the family $\mathcal A$ is
		faithful for Hausdorff dimension; that is, the Hausdorff
		dimension of every subset of $[0,1]$ is preserves when only covers by elements of $\mathcal A$. Set  $B:=\sup_{k\ge1}c_k<\infty$.
		
		\medskip
		\noindent
		\emph{Step 1.} We compare the lengths of corresponding fundamental intervals.
		
		Write
		\[
		\mathbb N\setminus\mathcal N
		=
		\{\nu_1<\nu_2<\cdots\}.
		\]
		Thus $\nu_n$ is the position occupied by the $n$-th original
		partial quotient after insertion. Equivalently,
		\[
		b_{\nu_n}=a_n,
		\qquad
		r(\nu_n)=\nu_n-n.
		\]

		Let
		\(
		U:
		=
		J_n(a_1,\ldots,a_{n-1};s,t)
		=
		\bigcup_{d=s}^{t}
		I_n(a_1,\ldots,a_{n-1},d)
		\in\mathcal A_n.
		\)
		Insert all prescribed digits $c_k$ whose positions lie in
		$\mathcal N\cap[1,\nu_n-1]$ into the fixed prefix
		$(a_1,\ldots,a_{n-1})$, and denote the resulting prefix by
		\(
		(b_1,\ldots,b_{\nu_n-1}).
		\)
		Since $\nu_n\notin\mathcal N$, the $n$-th original digit
		occupies position $\nu_n$ after insertion. Define
		\[
		\widetilde U
		:=
		\bigcup_{d=s}^{t}
		I_{\nu_n}(b_1,\ldots,b_{\nu_n-1},d)
		\in\mathcal A_{\nu_n}.
		\]
		By construction,
		\begin{equation}\label{3.2}
			\iota(E\cap U)\subset\widetilde U.
		\end{equation}
		
		Let $q_{n-1}$ be the denominator associated with the original
		prefix $(a_1,\ldots,a_{n-1})$, and let
		$\widetilde q_{\nu_n-1}$ be the denominator associated with
		the inserted prefix $(b_1,\ldots,b_{\nu_n-1})$.
		Hence
		\begin{equation}\label{3.3}
			q_{n-1}
			\leq
			\widetilde q_{\nu_n-1}
			\leq
			(B+1)^{r(\nu_n)}q_{n-1}.
		\end{equation}
		Proposition \ref{length of cylinder} gives
		\begin{equation}\label{3.4}
			|U|
			=
			\frac{t-s+1}
			{\bigl(sq_{n-1}+q_{n-2}\bigr)
				\bigl((t+1)q_{n-1}+q_{n-2}\bigr)}.
		\end{equation}
		The analogous formula holds for $\widetilde U$, with
		$q_{n-1},q_{n-2}$ replaced by
		$\widetilde q_{\nu_n-1},\widetilde q_{\nu_n-2}$.
		Combining \eqref{3.3} and \eqref{3.4} yields
		\begin{equation}\label{block-comparison}
			\frac{|U|}
			{4(B+1)^{2r(\nu_n)}}
			\leq
			|\widetilde U|
			\leq
			4|U|.
		\end{equation}
		
		\noindent
		\emph{Step 2.} We prove that $\dim_{\rm H}\iota(E)=\dim_{\rm H}E$.
		
		Let $\{J_{n_i}\}$ be a sufficiently fine $\mathcal A$-cover of
		$E$, where
		\(
		J_{n_i}\in\mathcal A_{n_i}.
		\)
		For each $J_{n_i}$, let $\widetilde J_{n_i}$ be the corresponding
		fundamental interval constructed in Step~1. By \eqref{3.2},
		\(
		\iota(E)\subset\bigcup_i\widetilde J_{n_i}.
		\)
		Therefore, by \eqref{block-comparison},
		for every $s>0$,
		\[
		\sum_i|\widetilde J_{n_i}|^s
		\leq
		4^s\sum_i|J_{n_i}|^s.
		\]
		Since $\mathcal A$ is faithful for Hausdorff dimension, it
		follows that
		\begin{equation}\label{upper-dim}
			\dim_{\rm H}\iota(E)
			\leq
			\dim_{\rm H}E.
		\end{equation}
		
		For the reverse inequality, we start with a fundamental interval
		\[
		\widetilde V
		=
		\bigcup_{d=s}^{t}
		I_n(b_1,\ldots,b_{n-1},d)
		\in\mathcal A_n
		\]
		such that
		\(
		\widetilde V\cap\iota(E)\neq\varnothing.
		\)
		We associate with $\widetilde V$ a fundamental interval
		$V\in\mathcal A$ covering its preimage under $\iota$.
		There are two cases.
		
		\smallskip
		\noindent
		\emph{Case 1: $n\notin\mathcal N$.}
		The $n$-th position corresponds to an original partial
		quotient. Deleting all prescribed inserted digits among
		$b_1,\ldots,b_{n-1}$ gives an original prefix and hence a
		fundamental interval $V\in\mathcal A$ such that
		\[
		\iota^{-1}
		\bigl(\widetilde V\cap\iota(E)\bigr)
		\subset V.
		\]
		The same comparison as in Step~1 gives
		\[
		|V|
		\leq
		4(B+1)^{2r(n)}|\widetilde V|.
		\]
		
		\smallskip
		\noindent
		\emph{Case 2: $n=N_k\in\mathcal N$.}
		At the $n$-th position every point of $\iota(E)$ has the
		prescribed digit $c_k$. Since
		\(
		\widetilde V\cap\iota(E)\neq\varnothing,
		\)
		and
		\(
		c_k\in[s,t].
		\)
		Hence
		\[
		\widetilde V\cap\iota(E)
		\subset
		\widetilde W,
		\]
		where
		\(
		\widetilde W
		:=
		I_n(b_1,\ldots,b_{n-1},c_k)
		\subset\widetilde V.
		\)
		After deleting all prescribed inserted digits from
		$\widetilde W$, we obtain a cylinder $V$ in the original
		continued fraction expansion such that
		\[
		\iota^{-1}
		\bigl(\widetilde V\cap\iota(E)\bigr)
		\subset V.
		\]
		Part (3) of Proposition \ref{length of cylinder}
		gives
		\[
		|V|
		\le
		2(B+1)^{2r(n)}|\widetilde W|
		\le
		2(B+1)^{2r(n)}|\widetilde V|.
		\]
		
		Thus, in either case, we can associate with every
		$\widetilde V\in\mathcal A_n$ meeting $\iota(E)$ a set
		$V\in\mathcal A$ such that
		\begin{equation*}
			\iota^{-1}\bigl(\widetilde V\cap\iota(E)\bigr)\subset V,
			\qquad
			|V|\le4(B+1)^{2r(n)}|\widetilde V|.
		\end{equation*}
		Since
		\(
		q_{n-1}\geq2^{(n-2)/2},
		\)
		we have
		\(
		|\widetilde V|
		\leq
		2^{-(n-2)}.
		\)
		Fix $\eta>0$. Since $r(n)=o(n)$, for all sufficiently large
		$n$,
		\[
		2r(n)\log(B+1)
		\leq
		\eta(n-2)\log2.
		\]
		Therefore, we obtain
		\begin{equation*}
			|V|
			\leq
			C_\eta|\widetilde V|^{1-\eta}
		\end{equation*}
		for every $\widetilde V\in\mathcal A$ meeting $\iota(E)$.
		
		Now let $\{\widetilde V_i\}\subset\mathcal A$ be a sufficiently
		fine cover of $\iota(E)$. By the preceding construction, there
		exist $V_i\in\mathcal A$ such that
		\[
		E\subset\bigcup_iV_i
		\quad\text{and}\quad
		|V_i|
		\leq
		C_\eta|\widetilde V_i|^{1-\eta}.
		\]
		Hence, for every $s>0$,
		\[
		\sum_i|V_i|^{s/(1-\eta)}
		\leq
		C_\eta^{s/(1-\eta)}
		\sum_i|\widetilde V_i|^s.
		\]
		Using again the faithfulness of $\mathcal A$, we conclude that
		\[
		\dim_{\rm H}E
		\leq
		\frac{\dim_{\rm H}\iota(E)}{1-\eta}.
		\]
		Letting $\eta\to0$ gives
		\begin{equation}\label{lower-dim}
			\dim_{\rm H}E
			\leq
			\dim_{\rm H}\iota(E).
		\end{equation}
		
		Finally, combining
		\eqref{upper-dim} and
		\eqref{lower-dim}, we obtain
		\[
		\dim_{\rm H}\iota(E)
		=
		\dim_{\rm H}E.
		\]
	\end{proof}
	\begin{prop}
		For every $0\leq\beta<\infty$,
		\[
		\dim_{\mathrm H}F_{0,\beta}
		=
		\frac{2}{2+\beta}.
		\]
	\end{prop}
	
	\begin{proof}
		For the upper bound, it is immediately that
		\[
		F_{0,\beta}\subset F_{\beta}.
		\]
		From \eqref{dim of Fbeta}, we have \[\hdim F_{0,\beta}\leq \hdim F_{\beta}\leq \frac{2}{\beta+2}.\]
		
		For the lower bound, choose a zero--density sequence
		$\mathcal N=\{N_k:k\geq1\}$ with $N_1\geq2$, and For
		$x=[a_1,a_2,\ldots]$, let
		$y=\iota_{1 }(x)=[b_1,b_2,\ldots]$ be obtained by inserting
		the digit $1$ at the positions in $\mathcal N$. Set
		\[
		\widetilde{F}_{\beta}
		:=
		\iota_{1}(  F_{\beta}).
		\]
		By Proposition \ref{lem-insertion}, we have
		\begin{equation*}
			\hdim \widetilde F_\beta
			=
			\hdim F_\beta
			=
			\frac{2}{2+\beta}.
		\end{equation*}
		We first verify that the insertion does not change the upper--limit condition.
		Follow the notion and argument of proof of Proposition \ref{lem-insertion}, we have
		\[
		\frac{
			\log\widetilde q_{\nu_{m+1}-1}(y)
		}{
			\log q_m(x)
		}
		\longrightarrow1.
		\]
		Consequently,
		\[
		\limsup_{m\to\infty}
		\frac{
			\log b_{\nu_{m+1}}(y)
		}{
			\log\widetilde q_{\nu_{m+1}-1}(y)
		}
		=
		\limsup_{m\to\infty}
		\frac{
			\log a_{m+1}(x)
		}{
			\log q_m(x)
		}
		=
		\beta.
		\]
		At the inserted positions, the corresponding numerator is
		\(\log1=0\). Hence
		\[
		\limsup_{n\to\infty}
		\frac{\log b_{n+1}(y)}
		{\log\widetilde q_n(y)}
		=
		\beta.
		\]
		
		It remains to verify the lower-limit condition. For $y\in\widetilde{F}_{\beta}$, when $n=N_k$, we have $a_{N_k}(y)=1$. Hence
		\[
		\lim\limits_{k\to\infty}\frac{\log a_{N_k}(y)}{\psi(N_k-1)}
		=0.
		\]
		It follows that
		\[
		\liminf_{n\to\infty}
		\frac{\log a_{n+1}(y)}{\psi(n)}
		=0.
		\]
		Therefore,
		\[
		\widetilde{F}_{\beta}
		\subset F_{0,\beta},
		\]
		and consequently
		\[
		\dim_{\mathrm H}F_{0,\beta}
		\geq
		\frac{2}{2+\beta}.
		\]
		This completes the proof.
	\end{proof}

	\subsection{The case $\alpha>0,\ 0\leq\beta<\infty
		~\text{and}~
		B_\psi>1+\beta$}
	By Proposition~\ref{empty}, $F_{\alpha,\beta}=\emptyset$.
	Hence $\dim_{\mathrm H}F_{\alpha,\beta}=0$.
	
	\subsection{The case $\beta=\infty$}
	
	\begin{prop}
		For every $0\leq\alpha\leq\infty$,
		\[
		\dim_{\mathrm H}F_{\alpha,\infty}=0.
		\]
	\end{prop}
	
	\begin{proof}
		
		If $x\in F_{\alpha,\infty}$, then
		\[
		\limsup_{n\to\infty}
		\frac{\log a_{n+1}(x)}{\log q_n(x)}
		=\infty.
		\]
		Hence, for every $M>0$,
		\[
		a_{n+1}(x)\geq q^M_n(x)
		\]
		for infinitely many $n\in\N$, and therefore
		\[
		F_{\alpha,\infty}\subset J_M,
		\]
		where $J_M$ defined in \eqref{JARNIK-1}.
		Therefore, we have
		\[
		\dim_{\mathrm H}F_{\alpha,\infty}
		\leq\dim_{\mathrm H}J_M\le
		\frac{2}{2+M}
		\]
		for every $M>0$. Letting $M\to\infty$ yields
		\[
		\dim_{\mathrm H}F_{\alpha,\infty}=0.
		\]
		
	\end{proof}

	\section{An relation to convergence and irrationality exponents}
	In this section, we conclude by relating $F_{\alpha,\beta}$ to $E_{\tau,v}$ and its sublevel counterpart $E_{\leq\tau,v}$.
	\begin{proof}[Proof of Corollary \ref{cor-Song}]
		Assume first that $0\leq\tau<\infty$. Set
		\[
		\psi(n)=\log(n+1),\qquad
		\alpha=\frac1\tau,\qquad
		\beta=v-2,
		\]
		with $\frac{1}{0}=\infty$. Then $B_\psi=1$. If
		$x\in F_{\alpha,\beta}$, the lower-limit condition implies
		$\tau(x)\leq\frac{1}{\alpha}=\tau$; this also holds when
		$\alpha=\infty$. By \eqref{irr-limsup}, $v(x)=v$.
		Thus $F_{\alpha,\beta}\subset E_{\leq\tau,v}$, and
		Theorem~\ref{thm} gives
		\[
		\dim_{\mathrm H}E_{\leq\tau,v}\geq\dim_{\mathrm H}F_{\alpha,\beta}\ge\frac1v.
		\]
		
		Conversely, $\tau(x)<\infty$ implies that
		$a_n(x)\to\infty$.
		By the same covering argument as that used in the proof of the upper
		bound in Case 1 of \cite[Theorem 1.1]{Song}, for $0<\varepsilon<v-2$, it gives
		\[
		E_{\leq\tau,v}\subset J_{v-2-\varepsilon}^{*}.
		\]
		It follows from Lemma~\ref{lem-4} that
		\[
		\dim_{\rm H} E_{\leq\tau,v}
		\leq \frac{1}{v-\epsilon}.
		\]
		Letting $\epsilon\to0$, gives the required upper bound. If $v=2$, by a result of Good \cite{Good41}, we have
		\[
		\dim_{\mathrm H}E_{\leq\tau,2}
		\leq
		\dim_{\mathrm H}\{x\in[0,1):a_n(x)\to\infty\}
		=\frac12.
		\]
		Finally, if $\tau=\infty$, then
		$E_{\leq\infty,v}=F_{v-2}$ and hence
		$\dim_{\mathrm H}E_{\leq\infty,v}=2/v$.
		This completes the proof.
	\end{proof}
	
	Although Theorem~\ref{thm} directly yields a dimension formula only for
	the joint sublevel set defined by $\tau(x)\leq \frac{1}{\alpha}$, the Cantor set $A_{\alpha,\beta}$
	constructed in Subsection \ref{cantor set} has a stronger property: every point has convergence exponent exactly
	$\frac{1}{\alpha}$ when $\psi(n)=\log(n+1)$.
	\begin{prop}
		Assume that
		\[
		\psi(n)=\log(n+1),\quad
		0<\alpha<\infty,\quad
		0\leq\beta<\infty.
		\]
		Then
		\[
		A_{\alpha,\beta}
		\subseteq
		\left\{
		x\in(0,1):
		\tau(x)=\frac1\alpha,\quad
		v(x)=2+\beta
		\right\}.
		\]
	\end{prop}
	\begin{proof}
		By Lemma~\ref{subset lem} and \eqref{irr-limsup}, $v(x)=2+\beta$ for every
		$x\in A_{\alpha,\beta}$.
		
		For the specific sequence $\lambda_n$ chosen in the construction,
		$\lambda_n\geq e^{1/\sqrt n}>
		\psi(n+1)/\psi(n)$ for all sufficiently large $n$.
		Hence the supremum defining $\ell_n$ is attained at $m=n$
		for all sufficiently large $n$, and therefore
		$L_n=(n+1)^\alpha$.
		
		For every $n$, the construction gives
		$a_{n+1}(x)\geq L_n\geq(n+1)^\alpha$.
		Thus $\sum_n a^{-s}_{n+1}(x)<\infty$ whenever
		$s>\frac{1}{\alpha}$, so $\tau(x)\leq\frac{1}{\alpha}$.
		
		On the other hand, for all sufficiently large
		$n\notin\mathcal N$, we have $u_n=L_n$ and hence
		\[
		a_{n+1}(x)
		<2\lceil(n+1)^\alpha\rceil
		\leq4(n+1)^\alpha.
		\]
		Since $N_k=r_{2^k}\geq2^k$, we have
		$\sum_{n\in\mathcal N}(n+1)^{-1}<\infty$. Consequently,
		\[
		\sum_{n=1}^{\infty}a_{n+1}(x)^{-1/\alpha}
		\geq
		4^{-1/\alpha}
		\sum_{n\notin\mathcal N}(n+1)^{-1}
		=\infty.
		\]
		Therefore, $\tau(x)\geq\frac{1}{\alpha}$.
		This completes the proof.
	\end{proof}
	
	\medskip
	\noindent\textbf{Acknowledgment.}
	W. Cheng is supported by the National Natural Science Foundation of China (No. 12371086, 12271175) and the National Key R\&D Program of China (No. 2024YFA1013700). J. Feng  is supported by National Natural Science Foundation of China (No. 12501114).

	{}
	
\end{document}